\documentclass[11pt]{amsart}

\usepackage{amssymb,amsmath,amsthm}

\newtheorem{thm}{Theorem}[section]
\newtheorem{thm*}{Theorem}
\newtheorem{lemma}[thm]{Lemma}

\newtheorem{lemma*}[thm*]{Lemma}
\newtheorem{prop}[thm]{Proposition}

\theoremstyle{definition}

\DeclareMathOperator{\Free}{Free}

\newcommand{\from}{\colon}

\newcommand{\N}{\mathbb{N}}

\renewcommand{\subset}{\subseteq}
\renewcommand{\supset}{\supseteq}

\newcommand{\baire}{\N^\N}
\newcommand{\cantor}{2^\N}
\newcommand{\ramsey}{[\N]^\N}

\newcommand{\restrict}{\restriction}

\newcommand{\union}{\cup}
\newcommand{\bigunion}{\bigcup}
\newcommand{\inters}{\cap}

\newcommand{\biginters}{\bigcap}

\newcommand{\define}[1]{\emph{#1}} %definitions in italics

\begin{document}

\thanks{C.C. was partially supported by NSF grant DMS-1500906. A.M. was partially supported by NSF grant DMS-1500974.}

\author{Clinton T.~Conley}
\address{Department of Mathematical Sciences, 
Carnegie Mellon University, Pittsburgh, PA 15213, U.S.A.}
\email{clintonc@andrew.cmu.edu}

\author{Andrew S.~Marks}
\address{Department of Mathematics, University of California at Los Angeles}
\email{marks@math.ucla.edu}

\title{Distance from marker sequences in locally finite Borel graphs}

\date{\today}

\maketitle

\section{Introduction}

The investigation of structure in marker sequences has been a recurring
theme of the study of countable Borel equivalence relations and Borel
graphs. Recall that if $E$ is a countable Borel equivalence relation on a
standard Borel space $X$, then we say that $A \subset X$ is a
\define{complete section} or \define{marker set} for $E$ if $A$ meets every
$E$-class. We similarly say that $A$ is a marker set for a locally
countable graph on $X$ if $A$ is a marker set for the connectedness
relation of $G$. Finally, a \define{marker sequence} $\{A_n\}_{n \in \N}$
for a countable Borel equivalence relation or graph is a countable sequence
of marker sets.

The
study of marker sequences forms the backbone of our understanding of
what groups generate
hyperfinite equivalence relations \cite{GJ} \cite{ScSe}, and the
combinatorics of Borel graphs generated by these group actions
\cite{GJKS}. More broadly, these ideas underlie many constructions in the
study of Borel graph combinatorics \cite{KM}. The first result about
markers was proved by Slaman and Steel \cite{SlSt} who showed that for every countable
Borel equivalence relation $E$ there is a decreasing sequence $A_0 \supset
A_1 \supset \ldots$ of markers with empty intersection $\biginters_n A_n =
\emptyset$.

Suppose $\Gamma$ is a finitely generated  group which acts on the space $2^\Gamma$
via the left shift action. Let $\Free(2^\Gamma)$ be the set of $x \in
2^\Gamma$ such that for all nonidentity $\gamma \in \Gamma$ we have $\gamma
\cdot x \neq x$, and let $G(\Gamma,2)$ be the graph on $\Free(2^\Gamma)$
where $x, y \in \Free(2^\Gamma)$ are adjacent if there is a generator
$\gamma$ of $\Gamma$ such that $\gamma \cdot x = y$. Let $d_{G(\Gamma,2)}$
be the graph distance metric for $G(\Gamma,2)$.
A recent result of Gao, Jackson, Krohne, and Seward states the following.

\begin{thm}[{\cite[Theorem 1.1]{GJKS}}]\label{thm:GJS}
Suppose $\Gamma$ is a finitely generated infinite group and $f \from \N \to
\N$ tends to infinity. Then for every Borel marker sequence $\{A_n\}_{n \in
\N}$ for $G(\Gamma,2)$, there exists an $x
\in \Free(2^\Gamma)$ such that for infinitely many $n$, we have 
$d_{G(\Gamma,2)}(x,A_n) < f(n)$
\end{thm}

This result led us to ask the following question: what can we say if the
function $f \from \N \to \N$ is allowed to vary depending on the point $x$?
Of course, we cannot possibly draw an analogous conclusion for an arbitrary 
Borel way of associating some $f_x \from \N \to \N$ to each
point $x$ in our space; given a Borel marker sequence $\{A_n\}_{n \in \N}$
for a graph $G$ on $X$, we could define $f_x(n) = d_G(x,A_n)$ for all $x
\in X$. Instead, we show the existence of some Borel map
$x \mapsto f_x$ for which we can draw a stronger conclusion than
that of Theorem~\ref{thm:GJS}, showing closeness for all $n$ instead of just infinitely
many $n$. This is true even
when we generalize to arbitrary locally finite non-smooth graphs.

\begin{thm}\label{thm:main}
  Suppose $G$ is a locally finite non-smooth Borel graph on $X$. Then there
  exists a Borel map associating to each $x \in X$ a function $f_x \from
  \N \to \N$ such that for every Borel marker sequence $\{A_n\}_{n \in \N}$ for
  $G$, there exists an $x \in X$ such that for all $n$, we have
  $d_G(x,A_n) < f_x(n)$.
\end{thm}

Now when $G$ is smooth, an easy diagonalization constructs marker sequences that
do not satisfy the conclusion of our theorem.
Hence, this result provides a novel way of
characterizing smoothness; a locally finite Borel graph is smooth if and
only if 
it does not admit Borel marker sequences that are somewhere ``far''
from every point.

We remark here that in Theorem~\ref{thm:main}, the map $x \mapsto f_x$ may
always be chosen so that it is a Borel
homomorphism from the equivalence relation graphed by $G$ to tail
equivalence on $\baire$. We note that standard ``sparse'' Borel marker
sequence constructions show that a map $x \mapsto f_x$ witnessing
Theorem~\ref{thm:main} cannot take a constant value on each connected
component of $G$. 

The theorem is proved in two steps. First, we use the fact that all Borel
subsets of $\ramsey$ are completely Ramsey to give an example of a Borel
graph $G$ satisfying the conclusion of Theorem~\ref{thm:main}. Then, we
conclude the full result using the Glimm-Effros dichotomy. We show in the
last section that this theorem cannot be proven using measure or category
arguments.

\section{Distance from marker sequences and the Ramsey property}

Let $\ramsey$ be Ramsey space, the space of infinite subsets of $\N$. Given a finite set $s \subset \N$ and an infinite
set $x \subset \N$ with $\min(x) > \max(s)$, recall the definition $[s,x]^\N = \{y \in
\ramsey : s \subset y \subset s \union x\}$.
We
can identify $\ramsey$ with a subset of $\cantor$ via characteristic
functions, and we use the resulting subspace topology on $\ramsey$
throughout. 
A theorem of Galvin and
Prikry~\cite{GP} states that for every $[s,x]^\N$ and every Borel subset $B \subset
[s,x]^\N$, there exists some $[t,y]^\N \subset [s,x]^\N$
such that either $[t,y]^\N \subset B$ or $[t,y]^\N \inters B = \emptyset$. 
From this, it is easy to see the following:
\begin{lemma}[Galvin-Prikry~\cite{GP}] \label{lemma:GP}
If 
$\{B_n\}_{n \in
\N}$ is a Borel partition of $[s,x]^\N$, then there exists some $n \in N$
and $[t,y]^\N \subset [s,x]^\N$ such that $[t,y]^\N \subset B_n$. 
\end{lemma}
\begin{proof}
  Suppose not. Then we may construct a decreasing sequence 
  $[s,x]^\N \supset [s_0, x_0]^\N \supset [s_1,x_1]^\N \supset \ldots$, where $[s_n,x_n]^\N
  \inters B_n = \emptyset$, and $s_n$ has at least $n$ elements. But then
  setting $z = \bigunion_n s_n$, we see that $z \in [s,x]$, and $z \notin
  B_n$ for all $n$, hence $\{B_n\}_{n \in \N}$ does not partition
  $[s,x]^\N$. 
\end{proof}

The \define{odometer} $\sigma \from \ramsey \to \ramsey$ is defined via the
identification of $\ramsey$ as a subspace of $\cantor$ by setting
$\sigma(x) = 0^n1y$ if $x = 1^n0y$, and fixing $\sigma(111\ldots) =
111\ldots$ on the sequence consisting of all ones. Define also $\tau \from \ramsey \to
\ramsey$ by setting $\tau(x) = \{n-1: n \in x \land n > 0\}$. Let $G_t$ be
the graph on $\ramsey$ generated by these two functions, where $x, y
\in \ramsey$ are adjacent if either $\sigma(x) = y$, $\sigma(y) = x$,
$\tau(x) = y$, or $\tau(y) = x$. So $G_t$ is graphing of tail equivalence
on $\ramsey$, and every vertex in $G_t$ has degree $\leq 5$. 

\begin{lemma}\label{lemma:E_t-graph}
  Consider the graph $G_t$ defined on $\ramsey$ as above, and 
  for each $x \in \ramsey$, let $f_x(n)$ be equal
  to the $(n+1)$st element of $x$. 
  Then for every Borel marker sequence $\{A_n\}_{n
  \in \N}$ for $G_t$, there is an $x \in \ramsey$ such that for all $n$, we
  have 
  $d_{G_t}(x,A_n) < f_x(n)$. 
\end{lemma}
\begin{proof}
  We construct $x$ as the intersection of a decreasing sequence $[s_0,
  x_0]^\N \supset [s_1,x_1]^\N \supset \ldots$, where $s_n$ has exactly
  $n$ elements. Let $s_0 = \emptyset$, and $x_0 = \N$. Now given
  $[s_n,x_n]^\N$, since the sets $\{\{y \in [s_n,x_n]^\N : d(y,A_n) =
  k\}\}_{k \in \N}$ partition $[s_n,x_n]$, we may apply
  Lemma~\ref{lemma:GP} to obtain some $[t,y] \subset
  [s_n,x_n]^\N$ and some $k$ such that every element of $[t,y]^\N$ is distance
  exactly $k$ from $A_n$. 
  Since $s_n \subset t$, there is
  some $m$ such that $m$ applications of the odometer applied to
  $[s_n,y]^\N$ yield
  $\sigma^m([s_n,y]^\N) = [t,y]^\N$. Hence by the triangle inequality, we see that there is some $k^* = k +
  m$ such that all the elements of $[s_n,y]$ are distance $\leq k^*$ from
  $A_n$. Now let $l$ be the least element of $y$ that is strictly greater than
  $k^*$ and $\max(s_n)$, let $s_{n+1} = s_n \union \{l\}$, and $x_{n+1}
  = y \setminus \{0, 1, \ldots, l\}$. We have ensured then that every
  element of $[s_{n+1},x_{n+1}]^\N$ has distance $\leq l$ from $A_n$.
\end{proof}

We remark here that the above proof works equally well for the usual
graphing of $E_0$ on $\ramsey$ induced by the odometer. We have used the
larger graph $G_t$ because we will need a locally finite graphing of tail
equivalence with our desired property in order to finish the proof of 
Theorem~\ref{thm:main}.

We need one more easy lemma before we complete the theorem. The lemma
roughly states that this question of closeness to marker sequences is
independent of the particular locally finite Borel graph we choose, and
depends only on the equivalence relation we have graphed.

Given a Borel graph $G$ on $X$, and a Borel map $x \mapsto f_x$ from $X \to
\baire$, say that a marker sequence $\{A_n\}_{n \in \N}$ \define{satisfies
$x \mapsto f_x$ for $G$} if for all $x \in X$ there exists an $n$ such that
$d_G(x,A_n) \geq f_x(n)$.

\begin{lemma}\label{lemma:graphing_invariant}
  Suppose $G$ and $H$ are locally finite Borel graphs on a standard Borel
  space $X$ having the same connected components. Then for every Borel
  map $x \mapsto g_x$ from $X \to \baire$, there exists a Borel map $x
  \mapsto h_x$ such that for every marker sequence $\{A_n\}_{n \in \N}$, if
  $\{A_n\}_{n \in \N}$ satisfies $x \mapsto h_x$ for $H$, then $\{A_n\}_{n
  \in \N}$ satisfies $x \mapsto g_x$ for $G$. 
\end{lemma}
\begin{proof}
  Since the graphs are locally finite, there are only finitely many points
  a fixed distance from each $x \in X$. Hence, we may define
  $h_x(n)$ to be the
  least $k$ such that $d_{H}(x,y) \leq k$ for all $y \in X$ such that
  $d_{G}(x,y) \leq g_x(n)$.
\end{proof}

We now complete the proof of Theorem~\ref{thm:main}.

\begin{proof}[Proof of Theorem~\ref{thm:main}]
  Suppose $G$ is not smooth. Let $E$ be the equivalence relation graphed by
  $G$, and $E_t$ be the equivalence relation of tail equivalence on
  $\ramsey$.
  By the Glimm-Effros dichotomy, there must be some $E$-invariant
  Borel set $A$ such that $E \restrict A \cong_B E_t$. But then $G
  \restrict A$ and the graph $G_t$ from Lemma~\ref{lemma:E_t-graph} are two
  different locally finite Borel graphings of the same equivalence
  relation. Hence, by Lemma~\ref{lemma:graphing_invariant}, we can find a
  Borel $x \mapsto h_x$ from $A
  \to \baire$ so that no Borel marker sequence can satisfy $G \restrict A$ for
  $x \mapsto h_x$.
  Hence, any Borel extension of $x \mapsto h_x$ to a function $x \mapsto
  f_x$ defined on $X$ suffices to prove the theorem. 
\end{proof}

\section{Measure and category}

In this section, we prove the following:
\begin{prop}
  Suppose $G$ is a locally finite Borel graph on $X$, and $x \mapsto f_x$
  is a Borel map from $X \to \baire$. Then 
  \begin{enumerate}
  \item For every Borel probability measure $\mu$ on $X$, there is a
  $G$-invariant $\mu$-conull set $B$ and a Borel marker sequence
  $\{A_n\}_{n \in \N}$ for $G \restrict B$ such that for every $x \in X$,
  there is an $n$ such that $d_G(x,A_n) \geq f_x(n)$. 
  \item For every compatible Polish topology $\tau$ on $X$, there is a
  $G$-invariant $\tau$-comeager set $B$ and a Borel marker sequence
  $\{A_n\}_{n \in \N}$ for $G \restrict B$ such that for every $x \in X$,
  there is an $n$ such that $d_G(x,A_n) \geq f_x(n)$. 
  \end{enumerate}
\end{prop}
\begin{proof}
  Let $B_0 \supset B_1 \supset B_2 \ldots$ be a decreasing sequence of
  Borel markers for $G$ with
  empty intersection. Such a sequence exists by [SlSt].
  Let $C_{i,n} = \{x \in X : d_G(x,B_i) < f_x(n)\}$. Note
  that since $\{B_n\}_{n \in \N}$ is decreasing with empty intersection,
  for each $n$, we have $\biginters_i C_{i,n} = \emptyset$.

  For part (1), we may assume as usual that $\mu$ is $G$-quasi-invariant.
  Observe that for each $n$, the $\mu$-measure of the sets $C_{i,n}$ goes
  to $0$. Hence, we may find a sequence $i_0, i_1, i_2, \ldots$ such that
  $\mu(C_{i_n,n}) \rightarrow 0$. Now choose our marker sequence to be
  $\{A_n\}_{n \in \N}$ where $A_n
  = B_{i_n}$. This marker sequence has the required property on the
  complement of the nullset $\biginters_{i} C_{i_n,n}$. 

  Part (2) follows using a similar argument, since relative to any basic
  open set, the set $C_{i,n}$ can be comeager for only finitely many $i$. 
\end{proof}

\end{document}